\documentclass{amsart}
\usepackage{amsfonts,amssymb,amscd,amsmath,enumerate,verbatim,calc}
\usepackage[all]{xy}
\usepackage{euscript}

\newcommand{\ff}{\text{if and only if}}
\newcommand{\wrt}{with respect to}
\newcommand{\fg}{finitely generated}
\newcommand{\eE}{\EuScript{E}}
\newcommand{\eC}{\EuScript{C}}
\newcommand{\n}{\mathfrak{n} }
\newcommand{\m}{\mathfrak{m} }

\newcommand{\fA}{\mathfrak{a}}

\newcommand{\Z}{\mathbb{Z} }
\newcommand{\pr}{\prime }
\newcommand{\fP}{\mathfrak{P}}
\newcommand{\fp}{\mathfrak{p}}
\newcommand{\fL}{\mathfrak{L} }
\newcommand{\aF}{\mathfrak{a} }
\newcommand{\R}{\mathcal{R}(I)}
\newcommand{\Sc}{\mathcal{S} }

\newcommand{\C}{\mathcal{C} }
\newcommand{\E}{\mathcal{E} }
\newcommand{\cV}{\mathcal{V}}

\newcommand{\bx}{ \mathbf{x}}

\newcommand{\bF}{ \mathbf{f}}

 \newcommand{\rt}{\rightarrow}
\newcommand{\xar}{\longrightarrow}

\newcommand{\wh}{\widehat }

\newcommand{\projdim}{\operatorname{projdim}}
\newcommand{\image}{\operatorname{image}}
\newcommand{\cx}{\operatorname{cx}}
\newcommand{\Spec}{\operatorname{Spec}}
\newcommand{\mSpec}{\operatorname{m-Spec}}

\newcommand{\grade}{\operatorname{grade}}

\newcommand{\Supp}{\operatorname{Supp}}
\newcommand{\ann}{\operatorname{ann}}
\newcommand{\spread}{\operatorname{spread}}
\newcommand{\ldim}{\operatorname{limDim}}

\newcommand{\Ass}{\operatorname{Ass}}
\newcommand{\Tor}{\operatorname{Tor}}
\newcommand{\Hom}{\operatorname{Hom}}
\newcommand{\Ext}{\operatorname{Ext}}

\theoremstyle{plain}

\newtheorem{thm}{Theorem}

\newtheorem{theorem}{Theorem}[section]

\newtheorem{lemma}[theorem]{Lemma}
\newtheorem{proposition}[theorem]{Proposition}
\newtheorem{question}[theorem]{Question}
\newtheorem{fact}[theorem]{Fact}

\theoremstyle{definition}
\newtheorem{definition}[theorem]{Definition}
\newtheorem{s}[theorem]{}
\newtheorem{remark}[theorem]{Remark}
\newtheorem{example}[theorem]{Example}

\theoremstyle{remark}

\begin{document}
\title{On the finite generation of a family of Ext modules}
\author{Tony ~J.~Puthenpurakal}
\date{\today}
\address{Department of Mathematics, Indian Institute of Technology Bombay, Powai, Mumbai 400 076}
\email{tputhen@math.iitb.ac.in}
\thanks{The work for this paper was done while the second
author was visiting University of Kentucky by a fellowship from Department of Science and Technology, India. The author is deeply grateful to DST for its financial support and University of Kentucky for its hospitality.}
 \begin{abstract}
Let $Q$ be a Noetherian ring with finite Krull dimension and let $\mathbf{f}= f_1,\ldots f_c$ be a regular sequence in $Q$.  Set $A = Q/(\mathbf{f})$. Let $I$ be an ideal in $A$, and  let $M$ be a  finitely generated $A$-module with $\projdim_Q M$ finite. Set $\R = \bigoplus_{n\geq 0}I^n$, the Rees-Algebra of $I$. Let $N = \bigoplus_{j \geq 0}N_j$ be  a finitely generated graded $\R$-module.
 We show that
\[
 \bigoplus_{j\geq 0}\bigoplus_{i\geq 0} \ \Ext^{i}_{A}(M,N_j)
\]
is a  finitely generated bi-graded module over $\Sc = \R[t_1,\ldots,t_c]$. We give two applications of this result to local complete intersection rings.
\end{abstract}
\dedicatory{Dedicated to Prof. L. L. Avramov on the occasion of his sixtieth birthday}
 \maketitle
 \section{introduction}
Let $A$ be a Noetherian ring. Let $I$ be an ideal in $A$  and let  $M$ be a finitely  generated $A$-module.
M. Brodmann \cite{Brod} proved that the set
  $\Ass_A  M/I^{n}M$ is independent of $n$ for all  large $n$.
This result is usually deduced by proving that \\
 $\Ass_A  I^n M /I^{n+1}M$  is independent of $n$ for all large $n$.

 \emph{Some Generalizations of Brodmann's result} \\
Fix $i \geq 0$. The following sets are independent of $n$ for all large $n$.
\begin{enumerate}[\rm 1]
  \item
(L. Melkerson and P. Schenzel) \cite[Theorem 1]{Melk-Sch}
\begin{enumerate}[\rm (a)]
  \item $\Ass_A  \Tor^{A}_{i}(M, I^n/I^{n+1})$.
  \item $\Ass_A  \Tor^{A}_{i}(M, A/I^{n})$.
\end{enumerate}
  \item  (same argument as in 1(a)).\\
    $\Ass_A  \Ext_{A}^{i}(M, I^n/I^{n+1})$.
  \item (D. Katz \& E. West; \cite[3.5]{Katz-West}) $\Ass_A  \Ext_{A}^{i}(M,  A/I^{n}A)$.
\end{enumerate}
An example of A. Singh \cite{Singh} shows that
 $$\Ass_A  \lim_{\rightarrow}\Ext_{A}^{i}(  A/I^{n}, M) \quad \text{ need not be
finite.} $$
 So in this example $$\bigcup_{n\geq 1} \Ass_A \Ext_{A}^{i}(  A/I^{n}, M) \quad \text{is not even finite}. $$

We state some questions in this area which motivated me.
  \begin{enumerate}
    \item (W. Vasconcelos: \cite[3.5]{Vascon98})
    Is the set
    \[
    \bigcup_{i\geq 0} \Ass_A \Ext^{i}_{A}(M, A)    \quad \text{finite ?}
    \]
    \item (L. Melkerson and P. Schenzel: \cite[page 936]{Melk-Sch})  Is the set
    \[
    \bigcup_{i\geq 0} \bigcup_{n\geq 0} \Ass_A \Tor^{A}_{i}(M, A/I^{n}) \quad \text{finite ?}
    \]
  \end{enumerate}

The motivation for the main result of this paper came from a Vasconcelos's question.
The author now does not believe that Vasconcelos's question has a positive answer in this generality. However he is unable to give a counter-example.
Note that  if $A$ is a Gorenstein local ring then  Vasconcelos's question has, trivially,  a positive answer. If we change the question a little then we may ask:
if $M$, $D$ are two finitely generated $A$-modules then is  the set
 \[
    \bigcup_{i\geq 0} \Ass_A \Ext^{i}_{A}(M, D)    \quad \text{finite ?}
    \]
 This  is not known for Gorenstein rings in general.
 Using Melkerson and Schenzel's question as a guidepost the questions I was interested to solve was

Let $(A,\m)$ be a local complete intersection of codimension $c$. Are the sets
\begin{enumerate}[\rm (a)]
\item
$\displaystyle{    \bigcup_{i\geq 0} \bigcup_{j\geq 0} \Ass_A \Ext_{A}^{i}(M, D/I^{j}D) \quad \text{finite ?}}$
\item
$\displaystyle{    \bigcup_{i\geq 0} \bigcup_{j\geq 0} \Ass_A \Ext_{A}^{i}(M, I^jD) \quad \text{finite ?}}$
\end{enumerate}
 In Theorem \ref{motiv-ci}
we prove that (b) holds. I have been unable to verify whether (a) holds.

The main  result in this paper is
the following  regarding finite generation of a family of Ext modules. Let $\R = \bigoplus_{n\geq 0}I^nt^n$ be the Rees algebra of $I$.
\begin{thm}\label{main}
Let $Q$ be a Noetherian ring with finite Krull dimension and let $\mathbf{f}= f_1,\ldots f_c$ be a regular sequence in $Q$.  Set $A = Q/(\mathbf{f})$.  Let $M$ be a finitely generated $A$-module with $\projdim_Q M$ finite.
Let $I$ an ideal in $A$ and let $N = \bigoplus_{n\geq 0}N_n$ be a finitely generated $\R$-module.
Then
$$\E(N)= \bigoplus_{i \geq 0}\bigoplus_{n \geq 0}\Ext^{i}_{A}(M,N_n)$$
 is a
\emph{finitely generated} bi-graded
$\Sc = \R[t_1,\ldots, t_c]$-module.
\end{thm}
An easy consequence of this result is that (b) holds (by taking $N = \bigoplus_{n\geq 0}I^nD$);  see Theorem \ref{motiv}.
A complete local complete intersection ring is a quotient of a regular local ring mod a regular sequence. So in this case  (b) holds from Theorem \ref{motiv}.
The proof of (b) for local complete intersections in general is a little technical; see Theorem \ref{motiv-ci}

We next discuss a surprising consequence of Theorem \ref{main}.
Let $(A,\m)$ be a  local complete intersection of codimension $c$.
Let $M, N$ be two finitely generated $A$-modules. Define
\[
\cx_A(M, N) = \inf\left \lbrace b \in \mathbb{N}  \left\vert \right.   \varlimsup_{n \to \infty} \frac{\mu(\Ext^{n}_{A}(M,N))}{n^{b-1}}  < \infty \right \rbrace
\]
In this section \ref{section-app2} we prove, see Theorem \ref{comp-stable},  that
\begin{equation*}
\cx_A(M, I^jN)  \quad \text{ is constant for all} \  \ j \gg 0. \tag{$\dagger$}
 \end{equation*}

We now describe in brief the contents of this paper.
In section one we give a module structure to $\E(N)$ over $\Sc$ (as in Theorem \ref{main}). We also discuss a few preliminaries.
The local case of Theorem \ref{main} is proved in section 2 while the global case is proved in section 3. In section 4 we prove our results on asymptotic primes. In section 5 we prove ($\dagger$).

\textit{Acknowledgements:} The author thanks Prof. L. L. Avramov and Prof. J. Herzog for many  discussions regarding this paper
\section{module structure}
Let $Q$ be a Noetherian ring  and let $\mathbf{f}= f_1,\ldots f_c$ be a regular sequence in $Q$.  Set $A = Q/(\mathbf{f})$.  Let $M$ be a finitely generated $A$-module with $\projdim_Q M$ finite.  \emph{We will not change $M$ throughout our discussion.}
Let $I$ an ideal in $A$ and let $N = \bigoplus_{n\geq 0}N_n$ be a finitely generated $\R = \bigoplus_{n\geq 0}I^nt^n$-module.
Set
$$\E(N)= \bigoplus_{i \geq 0}\bigoplus_{n \geq 0}\Ext^{i}_{A}(M,N_n).$$
 In this section we show $\E(N)$ is a bi-graded
$\Sc = \R[t_1,\ldots, t_c]$-module. We also discuss two preliminary results that we will need later in this paper.

\s Let $\mathbb{F}:  \cdots F_n \rt \cdots F_1 \rt F_0 \rt 0$
 be a free resolution of $M$ as a $A$-module.

 Let $t_1,\ldots t_c \colon \mathbf{F}(+2)  \rt \mathbf{F} $ be the \emph{Eisenbud-operators}
 \cite[section 1.]{Eisenbud-80} Then
 \begin{enumerate}
   \item $t_i$ are uniquely determined up to homotopy.
   \item $t_i, t_j$ commute up to homotopy.
 \end{enumerate}

\s   Set $T = A[t_1,\ldots, t_c]$ with $\deg t_i = 2$.  \\
In \cite{Gulliksen}  Gulliksen shows that if    $\projdim_Q M$ is finite then $\bigoplus_{i\geq 0}\Ext^{i}_{A}(M,L)$ is a finitely generated  $T$-module.

\s Let $N = \bigoplus_{n\geq 0} N_n$ be a f.g module over $\R$. Let $ u = xt^s $. The map
$$N_n \xrightarrow{u} N_{n+1} \quad \text{yields}$$
\[
\xymatrix
{
\Hom(\mathbf{F}, N_n) \ar@{->}[d]^{u}\ar@{->}[r]_{t_r}
&\Hom(\mathbf{F}, N_n)(+2) \ar@{->}[d]^{u}
\\
\Hom(\mathbf{F}, N_{n+s}) \ar@{->}[r]_{t_r}
&\Hom(\mathbf{F}, N_{n+s})(+2)
}
\]
Taking homology gives that

$\E(N)= \bigoplus_{i \geq 0}\bigoplus_{n \geq 0}\Ext^{i}_{A}(M,N_n)$ is a bi-graded $\Sc = R(I)[t_1,\ldots, t_c]$-module.

\begin{remark}
 1. For each $i$,  we have $ \bigoplus_{n \geq 0}\Ext^{i}_{A}(M,N_n)$ is a finitely generated $R(I)$-module.

2. For each $n$,  we have $ \bigoplus_{i \geq 0}\Ext^{i}_{A}(M,N_n)$ is a finitely generated $A[t_1,\ldots,t_c]$-module.

\end{remark}
We state two Lemma's which will help us in proving Theorem \ref{main}.

\s \label{not-sec2} \textbf{Notation}\\
(1) Let $N = \bigoplus_{n\geq 0}N_n$ be a graded $\R$-module. Fix $j \geq 0$. Set
\[
N_{\geq j} = \bigoplus_{n\geq j} N_n.
\]
$\E(N_{\geq j})$ is naturally isomorphic to the submodule  $$\E(N)_{\geq j} = \bigoplus_{i\geq 0} \bigoplus_{n\geq j}\E(N)_{ij}$$ of $\E(N)$.

(2) If $A\rt A^\pr$ is a ring extension and if $D$ is an $A$-module then set $D^\pr = D\otimes_A A^\pr$.  Notice that if $D$ is finitely generated $A$-module then $D^\pr $ is a finitely generated $A^\pr$-module.

(3) Set $\Sc^\prime = \Sc\otimes_A A^\pr$. Notice $\Sc^\pr$ is a finitely generated bi-graded $A^\pr$-algebra. Let $U =  \bigoplus_{i \geq 0}\bigoplus_{n \geq 0} U_{i,n} $
be a graded $\Sc$-module. Then
\[
U^\pr = U\otimes_A A^\pr =   \bigoplus_{i \geq 0}\bigoplus_{n \geq 0}  U_{i,n}^\pr
\]
is a graded $\Sc^\pr$-module.

\begin{lemma}\label{GEQj}
If $\E(N_{\geq j})$ is a finitely generated $\Sc$-module then $\E(N)$ is a finitely generated $\Sc$-module.
\end{lemma}
\begin{proof}
Set $D = \E(N)/ \E(N_{\geq j})$
We have the following exact sequence of $\Sc$-modules
\[
0 \xar \E(N_{\geq j}) \xar \E(N) \xar D \xar 0.
\]
Using Gulliksen's result it follows that $D$ is a finitely generated $T = A[t_1,\ldots,t_c]$-module. Since $T$ is a subring of $\Sc$, we get that
$D$ is a finitely generated $\Sc$-module.
Thus if $\E(N_{\geq j})$ is a finitely generated $\Sc$-module then $\E(N)$ is a finitely generated $\Sc$-module
\end{proof}
\begin{lemma}\label{flat}[with notation as in \ref{not-sec2}(3)]
Let $A \rt A^\pr$ be a faithfully flat extension of rings and let $U =  \bigoplus_{i \geq 0}\bigoplus_{n \geq 0} U_{i,n} $  be a graded $\Sc$-module
such that $U_{i,n}$ is a finitely generated $A$-module for each $i,n \geq 0$.
If
   $U^\pr$ is a finitely generated $\Sc^\pr$-module.
then $U$ is a finitely generated $\Sc$-module.
\end{lemma}
\begin{proof}
Choose a finite generating set $L$ of $U^\pr$. Suppose
$\deg_i u \leq a $ and $\deg_n u \leq b$ for each $u \in L$. For $0 \leq i \leq a$ and
$0\leq n \leq b$ choose a finite generating set $C_{i,n}$ of $U_{i,n}$. Set
\[
C = \bigcup_{i = 0}^{a}\bigcup_{n = 0}^{b} C_{i,n}.
\]
Set $C^\pr = \{ u\otimes 1 \mid u \in \C \}$.  Let $V$ be the submodule of $U$ generated by $C$. By construction $C^\pr$ generates $U^\pr$.
So $U^\pr = V^\pr$. Thus $(U/V) \otimes_A A^\pr = 0$. Since $A^\pr$ is a faithfully flat $A$-algebra we get $U = V$. So $U$ is a finitely generated $\Sc$-module.
\end{proof}
\section{The local case}
In this section we  prove Theorem \ref{main} when  $(Q,\n)$ is local. Let $\m$ be the maximal ideal of $A$. Set $k = A/\m$. Let $I$ be an ideal in $A$.
Set $F(I) = \R\otimes_A k = \bigoplus_{n\geq 0}I^n/\m I^n$ the \emph{fiber cone} of $I$.

\s Assume $N = \bigoplus_{n\geq 0}N_n$ is a finitely generated $\R$-module. Notice
$$F(N) = N \otimes_A k = \bigoplus_{n\geq 0}N_n/\m N_n $$
 is a finitely generated $F(I)$-module. Set
$$ \spread (N) := \dim_{F(I)}  N/\m N  \quad \text{the \emph{analytic spread} of}\  N. $$

\begin{proof}[Proof of Theorem \ref{main} in the local case]
\ { }  \\
Case 1. \emph{The residue field $k = A/\m$ is infinite.}

We induct on $\spread(N)$.

First assume $\spread(N) =0 $. This implies that $N_n/\m N_n = 0$ for all $n \gg 0$. By Nakayama Lemma, $N_n = 0$ for all $n \gg 0$; say
$N_n = 0$ for all $n \geq j$. Then $\E(N_{\geq j}) = 0$ and it is obviously a finitely generated $\Sc$-module. By \ref{GEQj} we get
that $\E(N)$ is a finitely generated $\Sc$-module.

 When $\spread(N) > 0$ then there exists $u = xt \in \R_1$ which is $N \oplus F(N)$-filter regular, i.e.,
 there exists $j$ such that
 $$(0 \colon_N u)_n = 0 \quad\text{and} \quad (0 \colon_{F(N)} u) = 0  \quad \text{for all} \ n \geq j. $$

 Set $N_{\geq j} = \bigoplus_{n\geq j}N_n$ and $U = N_{\geq j}/u N_{\geq j}$.  Notice we have an exact sequence of $\R$-modules
 \[
 0 \xar N_{\geq j}(-1) \xrightarrow{u} N_{\geq j}  \xar U \xar 0.
 \]
For each $n \geq j$ the functor $\Hom_A(M,-)$  induces the following long exact sequence of $A$-modules
\begin{align*}
0 &\xar \Hom_A(M,N_n) \xrightarrow{u} \Hom_A(M,N_{n+1}) \xar \Hom_A(M,U_{n+1}) \\
&\xar \Ext_A^1(M,N_n) \xrightarrow{u} \Ext^1_A(M,N_{n+1}) \xar \Ext^1_A(M,U_{n+1}) \\
&\cdots \quad \cdots \quad \cdots \quad \cdots \quad \cdots\quad \cdots \quad \cdots\quad \cdots \quad \cdots\\
&\xar \Ext_A^i(M,N_n) \xrightarrow{u} \Ext^i_A(M,N_{n+1}) \xar \Ext^i_A(M,U_{n+1}) \\
&\cdots \quad \cdots \quad \cdots \quad \cdots \quad \cdots\quad \cdots \quad \cdots\quad \cdots \quad \cdots
\end{align*}
Using the naturality of Eisenbud operators we have the following exact sequence of $\Sc$-modules
\[
\E(N_\geq j)(-1,0) \xrightarrow{(u,0)}  \E(N_\geq j) \xar \E(U)
\]
By construction
\[
\spread(U) = \spread(N_{\geq j}) - 1 = \spread(N)-1.
\]
By induction hypothesis $\E(U)$ is a finitely generated $\Sc$-module. Therefore by Lemma \ref{mod-u}
we get  $\E(N_{\geq j})$ is a \fg \  $\Sc$-module. Using \ref{GEQj} we get that $\E(N)$ is finitely generated $\Sc$-module.

Case 2.\emph{ The residue field $k$ is finite. }

In this case we do the standard trick.
 Let $Q^\pr = Q[X]_{\n Q[X]}$. Set $A^\pr = A \otimes_Q Q^\pr$. Notice $A^\pr = Q[X]_{\m Q[X]}$ is a flat $A$-algebra with residue field $k(X)$ which is infinite. Set $I^\pr = IA^\pr$ and $M^\pr  = M \otimes_{Q} Q^\pr = M \otimes_A A^\pr$.   Notice $\projdim_{Q^\pr} M^\pr$ is finite. Set $\R^\pr = \mathcal{R}(I^\pr)$ the Rees algebra of $I^\pr$. Notice that $N^\pr = N\otimes_A A^\pr $ is a finitely generated $\R^\pr$-module. Also
note that $\E(N^\pr) = \E(N)\otimes_A A^\pr$.

By Case 1 we have that $\E(N^\pr)$ is a \fg \ $\Sc^\pr$-module. So by Lemma \ref{flat} we get that $\E(N)$ is a \fg \ $\Sc$-module.
\end{proof}

The next Lemma is a bi-graded version of Lemma 2.8 (1) from \cite{Pu2}.
\begin{lemma}\label{mod-u}
Let $R$ be a Noetherian ring (not necessarily local) and let $B = \bigoplus_{i,j \geq 0}B_{i,j}$ be a finitely generated bi-graded
$R$-algebra with $B_{0,0} =R$. Set
$$ B_x = \bigoplus_{i\geq 0}B_{(i,0)} \quad\text{and} \quad B_y = \bigoplus_{j\geq 0}B_{(0,j)}. $$
 Let $V =  \bigoplus_{i,j \geq 0}V_{i,j}$ be a bi-graded $B$-module such that
\begin{enumerate}[\rm (1)]
  \item $V_{i,j}$ is a \fg \ $R$-module for each $i,j \geq 0$.
  \item For each $i\geq 0$,   $V_i = \bigoplus_{j\geq 0} V_{i,j}$ is \fg \ as a $B_y$-module
  \item  For each $j\geq 0$,   $V_j = \bigoplus_{i\geq 0} V_{i,j}$ is \fg \ as a $B_x$-module
  \item There exists $z \in B_{(r,0)}$ (with $r \geq 1$) such that we have the following exact sequence of  $B$-modules
  $$ V(-r,0) \xrightarrow{z} V \xrightarrow{\psi} D $$
  where $D$ is a \fg \ bi-graded $B$-module.
  \end{enumerate}
  Then $V$ is a finitely generated $B$-module
\end{lemma}
\begin{proof}
\emph{Step 1.} We begin by reducing to the case when $\psi$ is \emph{surjective}. \\
Notice $D^\pr = \image \psi$ is a \fg \ bi-graded $B$-module. If $\psi^\pr \colon V \rt D^\pr$ is the map induced by $\psi$ then we have an exact
sequence
\[
 V(-r,0) \xrightarrow{z} V \xrightarrow{\psi^\pr} D^\pr\xar 0.
\]
Thus we may assume $\psi$ is surjective.

\emph{Step 2.} Choosing generators:\\
\emph{2.1}: Choose a \emph{finite} set $W$ in $V$ of homogeneous elements such that
$$ \psi(W) = \{ \psi(w) \mid w\in W \}$$
is a \emph{generating set} for $D$. \\
\emph{2.2}: Assume all the elements in $W$ have $x$-co-ordinate $\leq c$. \\
\emph{2.3}: For each $i \geq 0$, by hypothesis $V_i$ is a \fg \ $B_y$-module. So we may choose a \emph{finite set} $P_i$ of homogeneous elements in $V_i$ which generates $V_i$ as a $B_y$-module. \\
\emph{2.4}: Set
\[
G = W \bigcup \left( \bigcup_{i =0}^{c} P_i \right).
\]
Clearly $G$ is a \emph{finite} set.

\emph{Claim:} $G$ is a generating set for $V$. \\
Let $U$ be the $B$-submodule of $V$ generated by $G$.  It suffices to prove that $U_{i,j} = V_{i,j}$ for all $i,j \geq 0$.
By construction we have that for $ 0\leq i \leq c$
\begin{equation*}
U_{i,j} = V_{i,j} \quad \text{for each }  j\geq 0 \tag{*}
\end{equation*}

Let $\preceq$ be the\emph{ lex-order} on $X = \Z_{\geq 0} \times \Z_{\geq 0}$. It is well-known that $\preceq$ is a total order on $X$. So we can prove our result by induction on $X$ \wrt \  the total order $\preceq$.

The base case is $(0,0)$.   \\
In this case $U_{0,0} = V_{0,0}$ by (*).

Let $(i,j) \in X\setminus \{ (0,0) \}$ and assume that for all $(r,s) \prec (i,j)$; we have  $U_{i,j} = V_{i,j}$.

Subcase 1. \ \  $i \leq c$. \\
By (*) we have $U_{i,j} = V_{i,j}$.

Subcase 2. \  \ $i > c$.

Let $p \in V_{i,j}$.
By construction,  note that there exists $w_1, \ldots, w_m \in W \subseteq C$ such that
\[
\psi(p) = \sum_{l = 0}^{m} h_l \psi(w_l) \quad \text{where $h_i \in B$}.
\]
We may assume that $\deg h_l w_l = (i,j)$ for each $l$
Set $p^\pr = \sum_{i = 0}^{m} h_i w_i \in V_{i,j}$.
Notice \\
1. $p^\pr \in U_{i,j}$. \\
2. $p - p^\pr \in \ker \psi$. \\
 So
\[
p - p^\pr = z\bullet q \quad \text{where} \ q \in V_{(i-r,j)}.
\]
If $q =0 $ then $p = p^\pr \in U_{i,j}$. \\
Otherwise note that $ (i-r,j) \prec (i,j)$. So by induction hypothesis
$q \in U_{(i-r,j)}$. It follows that $p \in U_{i,j}$.

Thus $V_{i,j} \subseteq U_{i,j}$. Since $U_{i,j} \subseteq V_{i,j}$ by construction it follows that
$U_{i,j} = V_{i,j}$.
The result follows by induction on $X$.
\end{proof}

 \section{The global case:}
We need quite a few preliminaries to prove the global case.
See \ref{diff} for the difficulty in going from local to the global case.
 Note that in the local case we proved the result by inducting on $\spread(N)$. This is
unavailable to us in the global situation as there are usually infinitely many maximal ideals in a global ring.
Most of this section we will discuss    two invariants of a graded $\R$-module  $N = \bigoplus_{n\geq 0}N_n$. We will use these invariants to prove Theorem
\ref{main} by induction.

\s\label{Nc-global} \emph{Notation and Conventions:}  We take dimension of the zero module to be $-1$. We define the zero-polynomial to have degree $-1$. \\ Let $\fP \in \Spec Q$.  If $\fP \supseteq \bF$ then set
$\fp = \fP/\bF$. If $\fP \nsupseteq \bF$ then any $A$-module localized at $\fP$ is zero. So assume $\fP \supseteq \bF$. Notice
\begin{enumerate}
  \item $\R_\fp \cong \mathcal{R}(IA_\fp)$ and $\Sc_\fp \cong \R_{\fp}[t_1,\ldots,t_c]$
  \item $M_\fp = M_\fP$ has finite projective dimension as a $Q_\fP$-module.
  \item $\E(N)_\fp \cong \E(N_\fp)$.
\end{enumerate}

\s \label{diff} \emph{The difficulty in going from local to global:} \\
For each $\fp \in \Spec A$ it follows from \ref{Nc-global} that $ \E(N_p)$ is a \fg \ $A_\fp$-module. Usually $\Supp_A \E(N)$ will
be an infinite set. So we cannot apply the local case and conclude.

The situation when  $\Supp_A \E(N)$ is a finite set will help in the  base step of our induction argument to prove Theorem \ref{main}. So we show it
separately.
\begin{lemma}\label{semi-local}
If $\Supp_A \E(N)$ is a finite set then $\E(N)$ is a \fg \ $\Sc$-module.
\end{lemma}
\begin{proof}
We may choose a finite subset $C$ of $\E(N)$ such that
its image in $\E(N)_\fp$ generates $\E(N)_\fp$ for each $\fp \in \Supp_A \E(N)$.
Set $U$ to be the \fg \ submodule of $\E(N)$ generated by $C$.

Set $D = \E(N)/U$. Notice that $D_\fp = 0 $ for each $\fp \in \Spec A$. So $D = 0$. Therefore $\E(N) = U$ is a
\fg\ $A$-module.
\end{proof}
\s  \label{ann} \emph{ First Inductive device:}  \\
Since $N$ is a \fg \ $\R$ -module we have
$\ann_A N_i  \subseteq \ann_A N_{i+1}$ for all $i \gg 0$. Since $A$ is Noetherian it follows that
$\ann_A N_n$ is constant for all $n \gg 0$. Call this stable value $\fL_N$.
This enables us to define
\emph{Limit dimension} of $N$.
$$\ldim N = \lim_{n \rt \infty} \dim_A N_n  = \dim A/\fL_N. $$
  Since $A$ has finite Krull-dimension we get that $\ldim N $ is finite.

\s \label{lim-ann-P} Let $\fP$ be a prime in ideal in $A$.
If $D$ is a finitely generated $A$-module then $$\ann_{A_\fP}D_\fP = \left( \ann_A D \right)_\fP = \left(\ann_A D \right)A_\fP. $$
Therefore
$$(\fL_N)_\fP = \fL_{N_\fP}. $$

 \s \label{case-init} Note that if $\ldim(N) = -1$ then $N_j = 0$ , say for all $j \geq j_0$. So $\E(N_{\geq j_0}) = 0$. Using \ref{GEQj} it follows that $\E(N)$ is a \fg \ $\Sc$-module. The first non-trivial case is the following
\begin{proposition}\label{case-zero}
If $\ldim(N) = 0$ then $\E(N)$ is a \fg \ $A$-module.
\end{proposition}
 \begin{proof}
 This implies that $A/\fL_N$ is Artinian.
  Say $\dim N_n = 0$ for $n \geq r$.
Clearly
$$\Supp_A \E(N_{\geq r}) \subseteq \Supp_A A/\fL_N \quad \text{  a finite set of maximal ideals in $A$}. $$
      It follows from \ref{semi-local} that   $\E(N_{\geq r})$ is a \fg \ $\Sc$-module.
Using \ref{GEQj} we get that $\E(N)$ is a \fg \ $\Sc$-module.
 \end{proof}

 \s\label{filt-deg-r}\textbf{higher degree filter-regular element}

  We do not have filter regular elements of degree $1$ in the global situation. However we  can do the following:

  Set $E = N/H^0_{R_+}(N)$.  Assume $E \neq 0$.  As $H^0_{R_+}(E) = 0$ there exists homogeneous  $u \in R_+ $ such that
  $u$  is $E$-regular, \cite[1.5.11]{BH}. Say $\deg u = s$. Since $E_n = N_n$ for all $n \gg 0$  it follows that the map
  $N_i \rt N_{i+s}$ induced by multiplication by $u$ is injective for all $i \gg 0$. We will say that $u$ is a $N$  filter-regular element
  of degree $s$.

\s \label{second}\emph{ The second inductive device:}  \\
We now discuss a global invariant of $N$ which patches well with  local ones.

\s \label{invar-local} \textbf{The local invariant} \\
Let $(A,\m)$ be local and let $W = \bigoplus_{n \geq 0} W_n$ be a finitely generated $\R$-module. For convenience we
 assume that $\fL_W = \ann_A W_n $ for all $n \geq 0$. Let $\fA \subseteq \fL_W$ be an ideal.
Fix $j \geq 0$. Set
\[
d_\fA( W, j) = \begin{cases}0, & \text{if $\dim W_j < \dim A/\fA$,} \\ e(\m, W_j);  &\text{otherwise.}                            \end{cases}
\]
Note that $W_j$ is an $A/\fA$-module and that $d_\fA( W, j)$ is the modified multiplicity function on the $A/\fA$-module $W_j$.

\begin{remark}\label{c-ideals}
Notice if $\dim W_j = \dim A/\aF$ then $$d_\fA( W, j) = d_{\fL_W}(W,j).$$
\end{remark}
Let $\mu(D)$ denote the minimal number of generators of an $A$-module $D$.
\begin{lemma}\label{order-lemma}
The function $d_\fA( W, -)$ is polynomial of degree $\leq \mu(I) - 1$.
\end{lemma}
\begin{proof}
We may assume that the residue field of $A$ is infinite.
Set $T = \R/\fA \R = \bigoplus_{n\geq 0}T_n$. Notice $T_0 = A/\fA$. Let $\bx = x_1,\ldots,x_r$ be a minimal reduction of $\m( A/\fA)$.  So $e(\m, -) = e(\bx,-)$, cf. \cite[4.6.5]{BH}.
 Then by
a result due to Serre, cf., \cite[4.7.6]{BH}, we get that
\[
e(\bx, W_j) = \sum_{i=0}^{r}(-1)^i\ell \left( H_i(\bx, W_j) \right)
\]
Notice  $H_i(\bx, W) = \bigoplus_{j \geq 0}  H_i(\bx, W_j)$ is a finitely generated $T/\bx T$-module. Notice
$(T/\bx T)_0 = A/(\fA + \bx)$ is Artinian. Furthermore  $(T/\bx T)_1$ is a quotient of $\R_1$ and so can be generated by
$\mu(I)$ elements.
Therefore the function  $j \mapsto  \ell \left( H_i(\bx, W_j) \right)$ is polynomial of degree $\leq \mu(I) - 1$. The result follows.
\end{proof}

\begin{definition}
$\theta(\fA, W) = $ degree of the polynomial function $d_\fA( W, -)$.
\end{definition}

Clearly $\theta(\fA, W)$
 is non-negative \ff \ $\ldim W = \dim R/\fA$ and is $-1$ otherwise.

\s \label{invar-global-local} \textbf{The global invariant} \\
Let $A$ be a Noetherian ring with finite Krull dimension. Let $I = (x_1,\ldots,x_s)$ be an ideal in $A$. Let $W = \bigoplus_{n \geq 0} W_n$ be a finitely generated $\R$-module. For convenience
we assume that $\fL_W = \ann_A W_n $ for all $n \geq 0$. Let $\fA \subseteq \fL_W$ be an ideal.

 Set
\[
\C(\fA) = \{ \m \mid \m \in \mSpec(A), \m \supseteq \fA \  \& \  \dim (A/\fA)_\m = \dim A/\fA \}.
\]
Let $I = (x_1,\ldots, x_s)$. If $\m \in \C(\fA)$ then note that
\begin{enumerate}[\rm (a)]
  \item $W_\m = \bigoplus_{n \geq 0} (W_n)_\m$.
  \item $\fL_{W_\m} = (\fL_W)_\m$. So $\fA_\m \subseteq \fL_{W_\m}$
  \item $\theta(\fA_\m, W_\m) \leq s -1$.
\end{enumerate}
Define
\[
\theta(\fA, W) = \max \{ \theta(\fA_\m, W_\m) \mid \m \in \C(\fA) \}.
\]
By (c) above we get that $\theta(\fA, W)$ is defined and is $\leq s- 1$.

\s\label{t-prop} \textbf{Properties of $\theta(\fA, W)$}. \\
We describe some properties of $\theta(\fA, W)$ we need for the proof of global case of Theorem \ref{main}.
Let $I = (x_1,\ldots, x_s)$.
\begin{enumerate}[\rm (i)]
  \item $\theta(\fA, W) \leq s-1$. \\
  This is clear.
  \item If $\fL_W \neq A$ then $\theta(\fL_W, W) \geq 0$. \\
  It suffices to consider the local case. Note that then $d_{\fL_W}( W, j) > 0$ for all $j \geq 0$. It follows that
  $\theta(\fL_W, W) \geq 0$.
  \item $\theta(\fA, W) = -1$ \ff \ $\ldim W < \dim A/\fA$.\\
  This follows from the following four facts:
  \begin{enumerate}[\rm (a)]
  \item
  By definition of $\C(\fA)$ we have that
  \[
  \dim A/\fA =  \dim (A/\fA)_\m \ \text{for each} \  \m \in \C(\fA).
  \]
    \item
    If $\theta(\fA, W) = -1$ then $\theta(\fA_\m, W_\m) = -1$ for all $\m \in \C(\fA)$. This is equivalent to saying that
    $\ldim W_\m < \dim (A/\fA)_\m$ for all $\m \in \C(\fA)$
    \item Note that since $\fA \subseteq \fL_W$ we get that
    \[
    \ldim W = \max \{ \ldim W_\m \mid \m \in \C(\fA) \}
    \]
  \end{enumerate}
  \item
  If $\theta(\fA, W) \geq 0$ then $\theta(\fL_W, W) \leq \theta(\aF, W)$. \\
  By previous item we get that $\ldim W = \dim A/\fA$. By hypothesis we also have
  $\aF \subseteq \fL_W$. Since $\dim A/\fA = \dim A/\fL_W$ it follows that
  $\C(\fL_W) \subseteq \C(\fA)$.   Using  \ref{c-ideals}
it follows that $\theta(\fL_W, W) \leq \theta(\fA, W)$.
\item
$u \in \R_+$ be homogeneous of degree $b$. Assume $u$ is $W$-filter regular. Set $E = W/uW$. Then
\[
\theta(\fL_W, E) \leq \theta(\fL_W, W) - 1.
\]
Suppose
$\theta(\fL_W, E) = \theta((\fL_W)_\fp, E_\fp)$ for some $\fp  \in \C(\fA)$. Since $u$ is $W$ filter regular; multiplication by $u$ induces the following exact sequence
\[
0 \rt W_{j-b} \rt W_j \rt E_j \rt 0 \quad \text{for all } \ j \gg 0.
\]
Localization at $\fp$ yields an exact sequence
\[
0 \rt (W_{j-b})_\fp \rt (W_j)_\fp \rt (E_j)_\fp \rt 0 \quad \text{for all } \ j \gg 0.
\]
Since $d_{{\fL_W}_\fp}(-, -)$ is an additive functor on $(A/\fL_W)_\fp)$-modules we get that
\[
\theta\left((\fL_W)_\fp, E_\fp \right) = \theta\left((\fL_W)_\fp, W_\fp \right) -1.
\]
The result follows since
\begin{align*}
\theta\left((\fL_W)_\fp,E_\fp \right) &= \theta(\fL_W, E)  \\
\text{and} \ \  \theta\left((\fL_W)_\fp, W_\fp \right) &\leq \theta(\fL_W, W).
\end{align*}
\end{enumerate}

We now give a proof of our main result
\begin{proof}[Proof of Theorem \ref{main}]
We induct on $\ldim N $.

If $\ldim N = -1, 0$ then the result follows from \ref{case-zero}.

Assume $\ldim N \geq 1$ and assume the result holds for all $\R$-modules $E$ with $\ldim E \leq \ldim N - 1$.
Let $x \in \R_+$ be homogeneous and a $N$-filter regular element. Set $D = N/xD$.
By Lemma \ref{GEQj} it suffices to assume the case when $x$ is $N$-regular.

We now induct on $\theta(\fL_N, N)$
If $ \theta(\fL_N, N) = 0$ then $\theta(\fL_N, D) \leq -1$, by \ref{t-prop}(v).  Using \ref{t-prop}(iii) we get that $$\ldim D < \dim A/\fL_N = \ldim N.$$
By induction hypothesis (on  $\ldim $) the module $\E(D)$ is finitely generated $\Sc$-module. The short exact sequence of $\R$-modules
\[
0 \rt N(-r)\xrightarrow{x} N \rt D \rt 0
\]
induces an exact sequence of $\Sc$-modules
\[
\E(N)(-r) \xrightarrow{x} \E(N) \rt \E(D).
\]
By Lemma \ref{mod-u} we get that $\E(N)$ is a finitely generated  $\Sc$-module.

We assume the result if $\theta(\fL_N, N)  \leq i$ and prove when $\theta(\fL_N, N) = i + 1$.
Let $D$ be as above. So $\theta(\fL_N, D) \leq i$, by \ref{t-prop}(v). \\
If $\theta(\fL_N, D) = -1$ then the argument as above yields $\E(N)$ to be a finitely generated $\Sc$-module.

If $\theta(\fL_N,D) \geq 0$ then by \ref{t-prop}(iv) we get that $\theta(\fL_D, D) \leq \theta(\fL_N, D) \leq i$. So by
induction hypothesis on $\theta(-,-)$ we get that $\E(D)$ is a finitely generated $\Sc$-module. By an argument similar to the
above we get that $\E(N)$ is a finitely generated $\Sc$-module.
\end{proof}

 \section{Application I \\ Asymptotic Associated primes}
In this section we give a proof of
our main motivating question Theorem \ref{motiv-ci}. We also give two proofs of Theorem \ref{motiv}.

\begin{theorem}\label{motiv}
Let $Q$ be a Noetherian ring with finite Krull dimension and let $\mathbf{f}= f_1,\ldots f_c$ be a regular sequence in $Q$.  Set $A = Q/(\mathbf{f})$.  Let $M$ be a finitely generated $A$-module with $\projdim_Q M$ finite.
Let $I$ an ideal in $A$ and let $N = \bigoplus_{n\geq 0}N_n$ be a finitely generated $\R$-module.  Then
\[
 \bigcup_{n\geq 0}\bigcup_{i\geq 0} \Ass \ \Ext^{i}_{A}(M,N_n)  \quad \text{is a finite set}.
\]
Furthermore there exists $i_0, n_0$ such that for all $ i \geq i_0$
and $n \geq n_0$ we have
\begin{align*}
 \Ass \  \Ext^{2i}_{A}(M,N_n) &= \Ass \  \Ext^{2i_0}_{A}(M,N_{n_0})   \\
\Ass \  \Ext^{2i+1}_{A}(M,N_n) &= \Ass \  \Ext^{2i_0 +1 }_{A}(M,N_{n_0})
\end{align*}
\end{theorem}
The following example   shows  that two set of stable values of associate primes can occur
\begin{example}
Let $ Q = k[[u,x]]$, $A = Q/(ux)$.
Let $M = Q/(u)$,   $I = A$ and $N = M[t]$ (so $N_n = M$ for all $n$).

For $i \geq 1$ one has
\[
\Ext^{2i-1}_{A}(M,M) = 0 \quad \Ext^{2i}_{A}(M,M) = k.
\]
\end{example}

  \s \label{west} We now state special case of a result due to E. West \cite[3.2;5.1]{West}.

 Let $R = A[x_1,\ldots, x_r ;y_1,\ldots y_s] $ be a  bi-graded $A$-algebra with
 $\deg x_i = (2,0)$ and $\deg y_j = (0,1)$. Let
  $M = \bigoplus_{i, n \geq 0}M_{(i,n)}$ be a finitely generated $R$-module. Then

  \begin{enumerate}
    \item $\bigcup_{i\geq 0}\bigcup_{n\geq 0} \Ass_A M_{(i,n)} $    is a finite set.
    \item  $ \exists \ i_0,  n_0$ such that for all $ i \geq i_0$
and $n \geq n_0$ we have
    \begin{align*}
    \Ass_A M_{(2i,n)}  &= \Ass_A M_{(2i_0,n_0)}   \\
    \Ass_A M_{(2i+1,n)}  &= \Ass_A M_{(2i_0 +1,n_0)}
    \end{align*}
  \end{enumerate}

 \begin{proof}[First proof of Theorem \ref{motiv}]
  The result follows from our main Theorem \ref{main} and \ref{west}.
  \end{proof}

For the convenience of the readers we give a self-contained second proof of Theorem \ref{motiv}.
We
need the following exercise problem from Matsumara's text (6.7, page 42) \cite{Mat}.

\begin{fact}\label{P-finite}
Let $f \colon A \rt B$ be a ring homomorphism of Noetherian rings. Let $U$ be a finitely generated $B$-module. Then
\[
\Ass_A U = \left \{ \fP\cap A \mid \fP \in \Ass_B U \right\}
\]
In particular $\Ass_A U$ is a finite set.
\end{fact}

  We will also need the following
\begin{lemma}\label{P-stable}
Let $R = A[x_1,\ldots, x_r ;y_1,\ldots y_s] $ be a  bi-graded $A$-algebra with
 $\deg x_i = (1,0)$ and $\deg y_j = (0,2)$. Let
  $D = \bigoplus_{i, j \geq 0}D_{(i,j)}$ be a finitely generated $R$-module. Set $l = r\bullet s$. Then there exists $i_0, j_0$ such that
  there exists inclusions
  \[
D_{(i,j)} \hookrightarrow D^l_{(i+1,j+2)} \quad \text{for all} \ \ i \geq i_0 \ \& \ j \geq j_0
 \]
 \end{lemma}
 \begin{proof}
Let $R_{++} = \bigoplus_{i\geq 1, j\geq 1}R_(i,j)$ be the irrelevant ideal of $R$. Set $$E = D/H^0_{R_{++}}(D).$$
It is easy to check that $D_{(i,j)} = E_{(i,j)} $ for all $i,j \gg 0$. So if $E = 0$ then we have nothing to prove.

We consider the case when $E \neq 0$. Notice that
\[
H^0_{R_{++}}\left(  E   \right) = 0. \quad \text{So} \ \grade(R_{++}, E) > 0.
\]
Now $R_{++} $ is generated by the $l$ elements $x_iy_j$; where $1\leq i \leq r$ and $1 \leq j \leq s$. Notice also that
$\deg x_iy_j = (1,2)$ for all $i,j$. Consider the map
\begin{align*}
\phi \colon E &\mapsto \left( E(1,2) \right)^l \\
d &\mapsto x_iy_j d.
\end{align*}
Since $\grade(R_{++}, E) > 0$ we get that $\phi$ is injective. Since $E_{(i,j)} = D_{(i,j)}$ for all $i,j \gg 0$; the result follows.
 \end{proof}

 \begin{proof}[Second proof of Theorem \ref{motiv}]
 Let $\Sc = \R[t_1,\ldots, t_c]$. Then $D = E(N)$ is a finitely generated $\Sc$-module.

 1. By \ref{P-finite}, it follows that $\Ass_A E(N)$ is a finite set. Notice that
 \[
 \Ass_A \E(N) = \bigcup_{n\geq 0}\bigcup_{i\geq 0} \Ass \ \Ext^{i}_{A}(M,N_n).
 \]

 2. By \ref{P-stable} it follows that we have inclusions
 \[
 \Ext^{i}_{A}(M,N_n) \hookrightarrow \Ext^{i+2}_{A}(M,N_{n+1}) \quad \text{for all} \ i,n \gg 0.
 \]
 It follows that
\[
 \Ass_A \Ext^{i}_{A}(M,N_n) \subseteq \Ass_A \Ext^{i+2}_{A}(M,N_{n+1}) \quad \text{for all} \ i,n \gg 0.
 \]
 Using 1. we get the result.
 \end{proof}

To prove an analog of Theorem \ref{motiv} for a local complete intersection we need the following result.
\begin{lemma}\label{AtoA}
Let $(A,\m)$ be a Noetherian local ring. Let $\wh A$ be the completion of $A$ \wrt \ $\m$. Let $B $ be a finitely generated $\wh A$-algebra containing $\wh A$. Let
$E$ be an $A$-module such that $E\otimes_A \wh A$ is a finitely generated $B$-module. Let $D$ be any finitely generated $ A$-module. Then
\begin{enumerate} [\rm (a)]
\item
$\displaystyle{  \Ass_{\wh A} E\otimes_A \wh A  \quad \text{is a finite set.  }       }$
\item
$\displaystyle{  \Ass_{ A} E\otimes_A \wh A  \quad \text{is a finite set.  }       }$
\item
$ \Ass_A E = \Ass_A E\otimes_A \wh A$. In particular $\Ass_A E$ is a finite set.
\item
$ \Ass_A D = \Ass_A D\otimes_A \wh A$.
\end{enumerate}
\end{lemma}
To prove this result we need Theorem 23.3 from \cite{Mat}.   Unfortunately  there is a typographical error in the statement of Theorem 23.3 in \cite{Mat}.
So we state it here.
\begin{theorem}\label{23}
Let $\varphi \colon A \rt B$ be a homomorphism of Noetherian rings, and let $E$ be an $A$-module and $G$ a $B$-module. Suppose that
$G$ is flat over $A$; then we have the following:
\begin{enumerate}[\rm (i)]
\item
if $\fp \in \Spec A$ and $G/\fp G \neq 0$ then
\[
 ^a \varphi \left( \Ass_B(G/\fp G)  \right) = \Ass_A (G/\fp G) = \{\fp \}.
\]
\item
$\displaystyle{\Ass_B(E\otimes G) = \bigcup_{\fp \in \Ass_A(E) } \Ass_B(G/\fp G).}$
\end{enumerate}
\end{theorem}
\begin{remark}
In  \cite{Mat} $\Ass_A(E\otimes G)$ is typed instead of $\Ass_B(E\otimes G)$.
\end{remark}
\begin{proof}[Proof of Theorem \ref{AtoA}]
We consider the natural ring  homomorphisms
\begin{align*}
\alpha &\colon A \hookrightarrow \wh A \\
\beta &\colon \wh A \hookrightarrow B \\
\gamma &\colon A \hookrightarrow  B
\end{align*}
Clearly $\gamma = \beta \circ \alpha$.

(a) We use the map $\beta$ and fact \ref{P-finite} to get our result.

(b) We use the map $\gamma$ and fact \ref{P-finite} to get our result.

(c) On the spectrum of the rings we have: $$ ^a \gamma = \  ^a\alpha \ \circ \ ^a\beta.$$
Therefore using  fact \ref{P-finite} we get
\begin{align*}
\Ass_A E \otimes \wh A &= \  ^a \alpha( \Ass_{\wh{ A}} E\otimes \wh A) \\
                       &= \{ P \cap A \mid P \in \Ass_{\wh{ A}} E\otimes \wh A  \}.
\end{align*}

   We consider the flat extension $ \alpha: A \rt  \wh A$.
By Theorem \ref{23}.(ii)  (with $G = \wh A$) we have
\begin{equation*}
\Ass_{\wh{A}} \left( E\otimes_A \wh A  \right) =  \bigcup_{\fP \in \Ass_A E} \Ass_{\wh A}  \wh A/\fP \wh A. \tag{*}
\end{equation*}
Now $\wh A$ is faithfully flat. So if $\fP \in \Spec A$ then $\wh A/ \fP \wh A \neq 0$. By  Theorem \ref{23}.(i) we get
\[
 ^a \alpha \left( \Ass_{\wh A}  \frac{\wh A} { \fP \wh A}\right)  = \fP
\]
Taking $ ^a\alpha$ on (*) yields the desired result.

(d) Set $B = \wh A$. The result follows from (c).
\end{proof}

We now prove the following:
\begin{theorem}\label{motiv-ci}
Let $(A,\m)$ be a local complete intersection.
Let $M$ be a finitely generated $A$-module.
Let $I$ an ideal in $A$ and let $N = \bigoplus_{n\geq 0}N_n$ be a finitely generated $\R$-module.  Then
\[
 \bigcup_{n\geq 0}\bigcup_{i\geq 0} \Ass \ \Ext^{i}_{A}(M,N_n)  \quad \text{is a finite set}.
\]
Furthermore there exists $i_0, n_0$ such that for all $ i \geq i_0$
and $n \geq n_0$ we have
\begin{align*}
 \Ass \  \Ext^{2i}_{A}(M,N_n) &= \Ass \  \Ext^{2i_0}_{A}(M,N_{n_0})   \\
\Ass \  \Ext^{2i+1}_{A}(M,N_n) &= \Ass \  \Ext^{2i_0 +1 }_{A}(M,N_{n_0})
\end{align*}
\end{theorem}
\begin{proof}
We consider the flat extension $\alpha \colon A \rt \wh A$. Say $\wh A = Q/(f_1,\ldots,f_c)$ where $(Q,\n)$ is a regular local ring and
$f_1,\ldots,f_c \in \n^2$ is a regular sequence.
1. Consider $ \E(N) = \bigoplus_{i\geq 0}\bigoplus_{n\geq 0}  \Ext^{i}_{A}(M,N_n)$ as an $A$-module. By Theorem \ref{main};  $E\otimes \wh A$ is a finitely generated $B = \mathcal{R}(I\wh A)[t_1,\ldots t_c]$-algebra. By
 Lemma \ref{AtoA} we get that $\Ass_A E$ is a finite set. Notice
\[
\Ass_A \E(N) = \bigcup_{n\geq 0}\bigcup_{i\geq 0} \Ass_A \ \Ext^{i}_{A}(M,N_n)
\]
2. Set $\E = \E(N)$. By \ref{P-stable} we have injective maps
\begin{equation*}
E_{(i,n)} \otimes \wh A \hookrightarrow \left( E_{(i+1, n+2)}\otimes \wh A \right)^l \quad \text{for all } \ i,n \gg 0. \tag{*}
\end{equation*}
Using Lemma \ref{AtoA}  with $B = \wh{A}$ and $E = E_{(i,n)}$ we get that for each $i\geq 0 $ and $n\geq 0$
\[
\Ass_A E_{(i,n)} = \Ass_A E_{(i,n)}\otimes_A \wh A
\]
The result follows from 1. and (*).
\end{proof}
 \section{Application II\\ Support Varieties}\label{section-app2}
Let $(A,\m)$ be a  local complete intersection of codimension $c$.
Let $M, N$ be two finitely generated $A$-modules. Define
\[
\cx_A(M, N) = \inf\left \lbrace b \in \mathbb{N}  \left\vert \right.   \varlimsup_{n \to \infty} \frac{\mu(\Ext^{n}_{A}(M,N))}{n^{b-1}}  < \infty \right \rbrace
\]
In this section we prove the following theorem

\begin{theorem}\label{comp-stable}
Let $(A,\m)$ be a local complete intersection, $M, N$  two finitely generated $A$-modules and let $I$ be a proper ideal in $A$. Then
\[
\cx_A(M, I^jN)  \quad \text{ is constant for all} \  \ j \gg 0.
 \]
\end{theorem}

\s\label{red-ci} \textbf{Reduction to the case when $A$ is complete and residue field of $A$ is algebraically closed}

\s \label{basic-flat} Suppose $A^\pr$ is a flat local extension of
$A$ such that $\m^\pr = \m A^\pr$ is the maximal ideal of $A^\pr$. If $E$ is an $A$-module then set $E^\pr = E\otimes_A A^\pr$.
Notice $I^\pr \cong IA^\pr$ and we consider it as an ideal in $A^\pr$. By \cite[7.4.3]{lucho98} $A^\pr$ is also a complete intersection.
It can be easily checked that
\[
\cx_{A^\pr}(M^\pr, (I^\pr)^jN^\pr) = \cx_A(M, I^jN) \quad \text{for all} \ n \geq 0.
\]
We now do our reduction in two steps\\

 By \cite[App. Th\'{e}or\'{e}me 1, Corollaire]{BourIX},  there exists a flat local extension
$A \subseteq \widetilde{A}$ such that $\widetilde{\m} = \m \widetilde{A}$ is the maximal ideal of $\widetilde{A}$ and the residue field $\widetilde{k}$ of $\widetilde{A}$ is an
algebraically closed extension of $k$.
By \ref{basic-flat} it follows that we may assume $k$ to be algebraically closed. We now complete $A$. Note that $\widehat{A}$ is a flat
extension of $A$ which satisfies \ref{basic-flat}.

Thus we may assume that our local complete intersection $A$
\begin{enumerate}[\rm(1)]
\item
 is complete. So $A = Q/(f_1,\ldots, f_c)$ where $(Q,\n)$ is regular local and $f_1,\ldots, f_c \in \n^2$ is a regular sequence.
\item
The residue field of $k$ is algebraically closed.
\end{enumerate}
\emph{Of course there exists many $Q$ and $f_1,\ldots, f_c$ of the type as indicated  above. We simply fix one such representation of $A$ }
\s Let $U, V$ be two finitely generated $A$-modules. \\ Let $\E(U,V) = \bigoplus_{n\geq 0} \Ext^{n}_A(U,V)$ be the total ext module of $U$ and $V$. We consider it as a
(finitely generated) module over the ring of cohomological operators $A[t_1,\ldots, t_c]$.
Since $\projdim_Q U$ is finite $\E^*(U,V)$ is a finitely generated
$A[t_1,\ldots, t_c]$-module.

\s Let $\eC(U,V) = \eE(U,V)\otimes_A k$. Clearly $\eC(U,V)$ is a finitely generated
  $T = k[t_1,\ldots, t_c]$-module. (Here degree of $t_i$ is 2 for each $i =1,\ldots, c$).
Set
\[
\fA(U,V) = \ann_T \eC(U,V).
\]
Notice that $\fA(U,V)$ is a homogeneous ideal.

\s We now forget the grading of $T$ and consider the affine space $\mathbb{A}^c(k)$.
Let
$$\cV(U,V) =  \cV\left(\fA(U,V)\right) \subseteq \mathbb{A}^c(k).$$
Since $\fA(U, V)$ is graded ideal we get that $\cV(U,V)$ is a cone.

\s\label{lucho-B}  By a result due to Avramov and Buchweitz \cite[2.4]{avr-b} we get that
\[
\dim \cV(U,V) = \cx_A(M,N)
\]

\begin{lemma}\label{cx}
If $I$ is an ideal in $A$ then there exists $j_0 \geq 0$ such that
\[
\cV(U, I^j V) = \cV(U, I^{j_0} V) \quad \text{for all} \ j \geq j_0.
\]
\end{lemma}
We give a proof of Theorem \ref{comp-stable}
assuming Lemma \ref{cx}.
\begin{proof}[Proof of Theorem \ref{comp-stable}]
By \ref{basic-flat} we may assume that $A$ is also complete and has an algebraically closed residue field.
The result now follows from \ref{lucho-B} and \ref{cx}.
\end{proof}

\s Let $N = \bigoplus_{j \geq 0} I^j V$. Set $M = U$. By hypothesis $A = Q/(f_1,\ldots, f_c)$ where
$(Q,\n)$ is a regular local ring and $f_1,\ldots, f_c \subseteq \n^2$ is a regular sequence.

\s Set $\eC(N) = \eE(N)\otimes_A  k$. By Theorem \ref{main} $\eE(N)$ is a finitely generated $\Sc = \R[t_1,\ldots,t_c]$-module.
It follows that $\eC(N)$ is a finitely generated, bi-graded,  $G = F(I)[t_1,\ldots, t_c]$-module.
Recall that $F(I)$, the fiber-cone of $I$, is a finitely generated $k$-algebra.
So we may as well consider $\eC(N)$ as a bi-graded

$R = k[X_1,\ldots, X_m, t_1,\ldots t_c]$-module (of course here
$X_1,\ldots ,X_m$ are variables). Furthermore $\deg X_l = (1,0)$ for $l = 1,\ldots m$ and
$\deg t_s = (0,2)$ for $s = 1,\ldots, c$. Set $T = k[t_1,\ldots,t_c]$.

\s\label{complexity-const} \textbf{Advantages of coarsening the grading on $\eC(N)$}

 By forgetting the degree on t's we  may consider $R = T[X_1,\ldots,X_m]$. Notice that the correspondingly we obtain
 \[
 \eC(N) = \bigoplus_{n\geq 0}\eC(U, I^jV)
 \]
We now give a
\begin{proof}[Proof of Lemma \ref{cx}]
 We make the constructions as in  \ref{complexity-const}. So $\eC(N)$ is a finitely generated
 $R = T[X_1,\ldots, X_m]$-module. Notice that $R$ is $\mathbb{N}$-standard graded. So there exists $j_0$ such that
$$\ann_T \eC(N)_j =  \ann_T \eC(N)_{j_0} \quad \text{for all } \ j \geq 0. $$
The results follows.
\end{proof}

\begin{question}[with hypothesis as in \ref{comp-stable}]
Is
\[
\cx_A(M, N/I^jN)  \quad \text{  constant for all} \  \ j \gg 0?
 \]
\end{question}

\providecommand{\bysame}{\leavevmode\hbox to3em{\hrulefill}\thinspace}
\providecommand{\MR}{\relax\ifhmode\unskip\space\fi MR }
\providecommand{\MRhref}[2]{%
  \href{http://www.ams.org/mathscinet-getitem?mr=#1}{#2}
}
\providecommand{\href}[2]{#2}

\end{document}